\documentclass{article}
\usepackage{amssymb,amsmath,amsthm,graphicx}

\def\qed{\hfill {\hbox{${\vcenter{\vbox{               
   \hrule height 0.4pt\hbox{\vrule width 0.4pt height 6pt
   \kern5pt\vrule width 0.4pt}\hrule height 0.4pt}}}$}}}

\def\utr{\, \underline{\triangleright}\, }
\def\otr{\, \overline{\triangleright}\, }

\newtheorem{theorem}{Theorem}

\newtheorem{proposition}[theorem]{Proposition}
\newtheorem{corollary}[theorem]{Corollary}

\theoremstyle{definition}
\newtheorem{example}{Example}
\newtheorem{definition}{Definition}
\newtheorem{remark}{Remark}

\date{}

\title{\Large \textbf{Biquandle Arrow Weight Enhacements}}

\author{Sam Nelson\footnote{Email: Sam.Nelson@cmc.edu. Partially supported by Simons Foundation collaboration grant 702597.}\and
Migiwa Sakurai\footnote{Email: migiwa@shibaura-it.ac.jp}}

\begin{document}
\maketitle

\begin{abstract}
We introduce a new infinite family of enhancements of the biquandle homset 
invariant called biquandle arrow weights. These invariants assign weights in an 
abelian group to intersections of arrows in a Gauss diagram representing a
classical or virtual knot depending on the biquandle colors associated to the 
arrows. We provide examples to show that the enhancements are nontrivial and
proper, i.e., not determined by the homset cardinality.
\end{abstract}

\parbox{6in} {\textsc{Keywords:} Biquandles, homsets, enhancements, virtual 
knots, Gauss diagrams, biquandle arrow weights

\smallskip

}

\section{\large\textbf{Introduction}}\label{I}

A combinatorial approach to knot theory involves representing knots as
equivalence classes of diagrams. The most common such scheme uses \textit{knot
diagrams}, 4-valent graphs with vertices enhanced with crossing information
so the graph can be interpreted as a projection of the knot in space onto 
the plane of the paper with only double-point singularities represented by 
the vertices. A knot can then be identified as an equivalence class of knot
diagrams under the equivalence relation generated by the Reidemeister moves 
and planar isotopy.

There are other combinatorial diagrammatic ways to represent knots, however,
such as the triple-point diagrams and \"uber-crossing diagrams in \cite{A} 
and the Gauss 
diagrams used in \cite{GPV}, the latter of which we use in this paper. 
Equivalence classes of Gauss diagrams under the Gauss diagram Reidemeister 
moves are known as \textit{virtual knots}, a larger category which includes 
classical knots as a subcategory. In papers such as \cite{C,OS}, invariants 
of virtual knots such as the affine index polynomial
are computed from Gauss diagrams using crossings of arrows, enigmatic features
of Gauss diagrams which don't correspond to anything obvious in the usual 
diagrams of knots. In this paper we adapt the idea from \cite{CJKLS} of 
assigning weights to crossings in diagrams colored with an algebraic structure,
in our case finite biquandles, such that the sum of the weights is preserved
by Reidemeister moves. The resulting multiset of weight values over the set of 
colorings of a Gauss diagram is therefore an invariant of virtual (and hence 
classical) knots for each finite biquandle and biquandle arrow weight. 

The paper is organized as follows. In Section \ref{R} we review the basics of 
biquandles and Gauss diagrams. In Section \ref{BW} we define biquandle arrow 
weight systems and define the new family of invariants, including our main 
result showing that the resulting enhancement of the counting invariant 
is in fact an invariant of oriented classical and virtual knots. In Section 
\ref{E} we collect some examples and in particular show that the enhancement 
is proper, i.e. not determined by the counting invariant. We conclude in 
Section \ref{Q} with some questions for future research. The first listed 
author thanks the second listed author and Shibaura Institute of Technology
for their kind hospitality during the preparation of this paper.

\section{\large\textbf{Biquandles and Gauss Diagrams}}\label{R}

We recall the definition of biquandles; see \cite{EN} for more.

\begin{definition}
A \textit{biquandle} is a set $X$ with operations 
$\utr$, $\otr : X \rightarrow X$ 
which satisfies for all $x, y, z  \in X$

\begin{itemize}
\item[(i)] $x \ \utr \ x = x \ \otr \ x$,
\item[(ii)]  the maps $\alpha_x$, $\beta_x : X \rightarrow X$ and $S : X \times X \rightarrow X \times X$ defined by
$$
\alpha_x(y) = y \ \otr \ x, \ \ \beta_x(y) = y \ \utr \ x \ \ and \ \ S(x, y) = (y \ \otr \  x, x \ \utr \ y)
$$
are invertible, and
\item[(iii)] we have the exchange laws
\begin{align*}
(x  \ \utr \  y)  \ \utr \  (z  \ \utr \  y) &= (x  \ \utr \  z) \ \utr \ (y  \ \otr \ z),\\
(x \ \utr \ y) \ \otr \ (z \ \utr \ y) &= (x \ \otr \ z) \ \utr \ (y \ \otr \ z)\mbox{, and}\\
(x  \ \otr \  y)  \ \otr \  (z  \ \otr \  y) &= (x  \ \otr \  z)  \ \otr \  (y \ \utr \ z).
\end{align*}
\end{itemize}
\end{definition}

A \textit{biquandle coloring} of an oriented knot or link diagram $D$ by a 
biquandle $X$, also known as an \textit{$X$-coloring}, is an assignment of 
an element of $X$ to each semiarc in $D$ such that at every crossing we have 
the condition
\[\scalebox{0.6}{\includegraphics{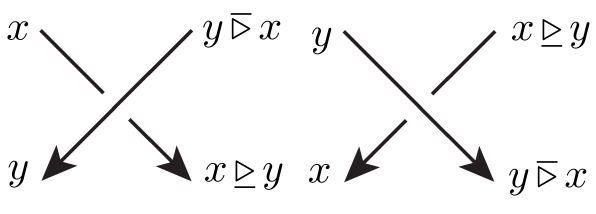}}\]
known as the \textit{biquandle coloring condition}.

The biquandle axioms are chosen so that for every biquandle coloring of a 
diagram before a Reidemeister move, there is a unique coloring of the diagram 
after the move which coincides with the original coloring outside the 
neighborhood of the move. It then follows that the number of colorings of an
oriented knot or link diagram by a finite biquandle $X$ is an integer-valued
invariant, denoted $\Phi_X^{\mathbb{Z}}$ and called the \textit{biquandle 
counting invariant}.

\begin{remark}\label{rem1}
More precisely, to every oriented knot or link $K$ there is an associated
biquandle called the \textit{fundamental biquandle} or the
\textit{knot biquandle}, denoted $\mathcal{B}(K)$, which can be described 
from any diagram of $K$ with a presentation consisting of generators 
associated to semiarcs and relations at the crossings. Any $X$-coloring 
then determines and is determined by a unique \textit{biquandle homomorphism}
$f\in\mathrm{Hom}(\mathcal{B}(K),X)$. This set of biquandle homomorphisms,
known as a \textit{homset}, is an invariant of $K$. Choosing a different
diagram for $K$ changes the representations of each homset element, with colored
diagrams representing the same homset element if they are related by $X$-colored
Reidemeister moves. We can think of choosing a diagram for $K$ as analogous 
to choosing input and output bases for vector spaces; then $X$-colored diagrams
represent homset elements analogously to the way matrices represent linear 
transformations, and Reidemeister moves play the role of change of basis 
matrices.
\end{remark}

Next, we recall the definition of Gauss diagrams.
Let $K$ be a virtual knot and $D$ a virtual knot diagram of $K$.  
Then, $D$ is regarded as the image $f(\mathbb{S}^1)$ of a generic immersion 
$f : \mathbb{S}^1$ $\to$ $\mathbb{R}^2$.  A {\it Gauss diagram} for $D$ is 
the preimage of $D$ with chords, each of which connects the preimages of 
each real crossing. We suppose that virtual knot diagrams are oriented. We 
specify over/under information of each real crossing on the corresponding 
chord by directing the chord toward the under path and assigning each chord 
with the sign of the crossing as shown: 
\[\scalebox{0.6}{\includegraphics{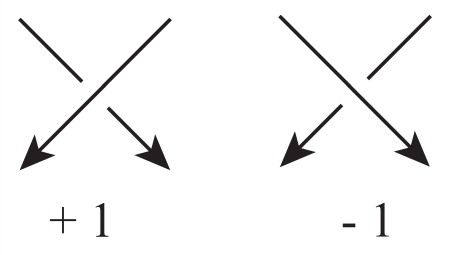}}\]

It is well-known that there exists a bijection from the set of virtual 
knots to the set of equivalence classes of their Gauss diagrams by the 
\textit{generalized Reidemeister moves of Gauss diagrams}: 
\[\includegraphics[width=8cm,clip]{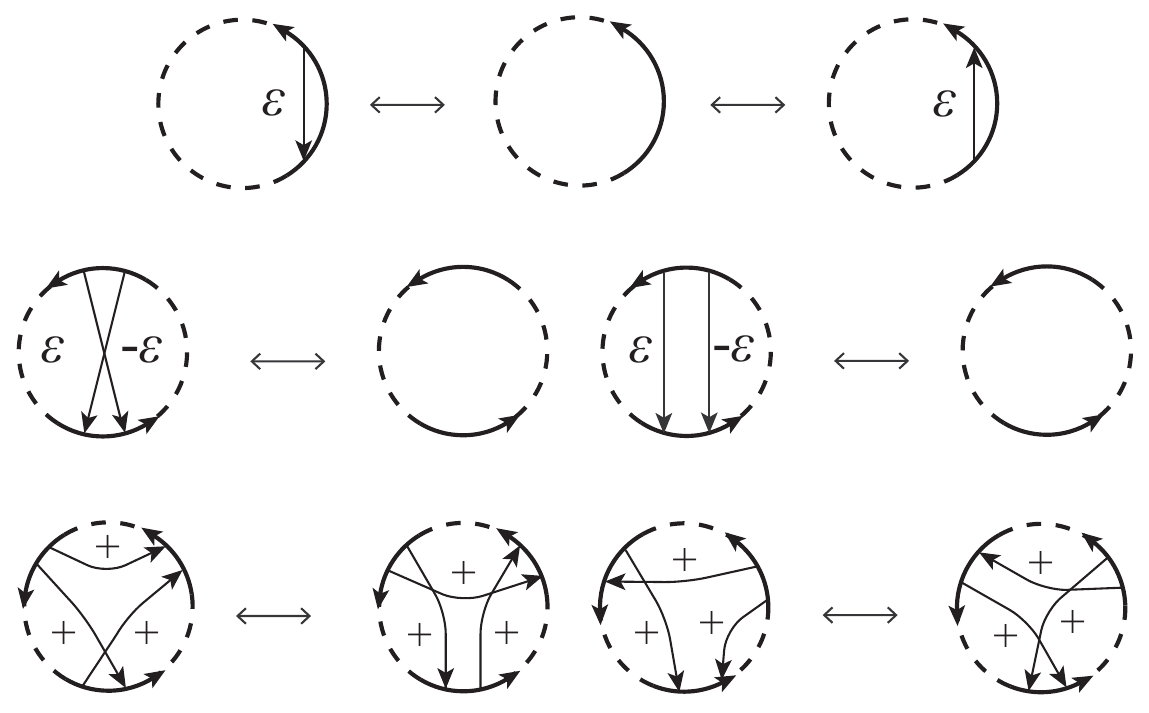}\]
We identify a  virtual knot with an equivalence class of Gauss diagrams 
under these moves.

A \textit{biquandle coloring} of a Gauss diagram by a biquandle $X$ is an 
assignment of an element of $X$ to each segment of the circle between the 
arrowheads and tails such that at every arrow we have
\[\scalebox{0.8}{\includegraphics{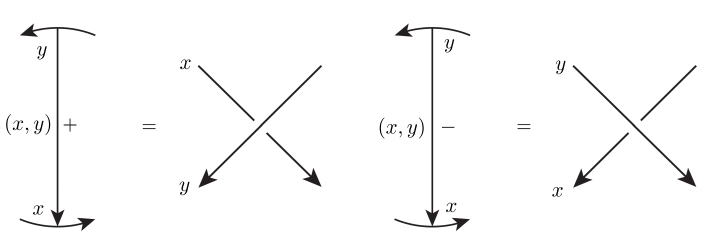}}\]
as shown. In particular, to each arrow in a biquandle-colored Gauss diagram
we associate an ordered pair $(x,y)$ of biquandle elements together with a 
sign.

As in Remark \ref{rem1}, the set of biquandle colorings of a Gauss diagram
$D$ representing a virtual knot $K$ by a biquandle $X$ can be identified 
with the homset $\mathrm{Hom}(\mathcal{B}(K),X)$, with each coloring 
representing a homset element and with our choice of diagram $D$ analogous to
a choice of basis for a vector space.

\section{\Large\textbf{Biquandle Arrow Weights}}\label{BW}

We begin with a definition.

\begin{definition}\label{def:BAW}
Let $X$ be a biquandle and $A$ an abelian group. A \textit{biquandle 
arrow weight} is a function $\phi:X^4\to A$ satisfying the
conditions
\begin{itemize}
\item[(i)] For all $x,y,u,v\in X$ we have 
\[\phi((x,y),(u,v))=\phi((u,v),(x,y)),\]
\item[(ii)] For all $x,y\in X$, $\phi((x,y),(x,y))=0,$
\item[(iii)] 
For all $x,y,z\in X$ we have
\[\phi((x,y),(y,z))=\phi((x,z),(y\otr x,z\otr x))+\phi((x,z),(x\utr z,y\utr z))\]
and
\item[(iv)]
For all $x,y,z\in X$ we have
\[\phi((x\utr z,y\utr z),(y\otr x,z\otr x))
=\phi((x,y),(x\utr y,z\otr y))+\phi((y,z),(x\utr y,z\otr y))\]
\item[(v)] For all $u,v,x,y,z\in X$ we have
\begin{eqnarray*}
\phi((u,v),(x,y))+\phi((u,v),(y,z))
& = & \phi((u,v),(x\utr z,y\utr z))+\phi((u,v),(y\otr x, z\otr x)) \\
\phi((u,v),(x,z))+\phi((u,v),(y\otr x, z\otr x)) 
& = & \phi((u,v),(y,z))+\phi((u,v),(x\utr y,z\otr y)) \\
\phi((u,v),(x,y))+\phi((u,v),(x\utr y,z\otr y)) 
& = & \phi((u,v),(x,z))+\phi((u,v),(x\utr z,y\utr z)). 
\end{eqnarray*}
\end{itemize}
\end{definition}

The motivation for this definition is the following: let $D$ be a Gauss diagram
with coloring by a biquandle $X$. Whenever two arrows with biquandle
colors $(x,y)$ and $(u,v)$ and signs $\epsilon$ and $\epsilon'$ cross, we 
want to assign a weight of $\epsilon \epsilon'\phi((x,y),(u,v))$ to the
crossing point. For example, in the case
\[\includegraphics{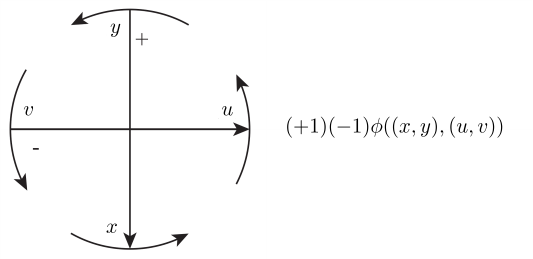}\]
we have a weight of $(+1)(-1)\phi((x,y),(u,v))=-\phi((x,y),(u,v))$.
The conditions in Definition \ref{def:BAW} are then the 
conditions needed for invariance of the sum of weights in a diagram under
the Gauss diagram Reidemeister moves. More precisely, we have:

\begin{proposition}\label{prop:main}
Let $X$ be a biquandle, $Y$ a biquandle arrow weight and $D$ a Gauss diagram
with a choice of $X$-coloring. Then the sum of $\phi$-values over all arrow 
crossings in $D$, denoted $\Sigma_D$, is not changed by $X$-colored 
Reidemeister moves.
\end{proposition}

\begin{proof}
This is a matter of checking the statement for a generating set of 
Reidemeister moves. We note that one such set consists of all four RI moves,
all four RII moves and the RIII move with all positive crossings; see \cite{P}.

First, we observe that there's no particular ordering on the set of arrows, 
so we set 
\[\phi((x,y),(u,v))=\phi((u,v),(x,y))\]
as axiom (i).

The RI moves involve introduction or removal of single arrows which do not
cross other arrows, so the contribution to weight sum the on both sides of 
the move is the same, namely zero, and the RI move conditions are
satisfied.
\[\scalebox{0.7}{\includegraphics{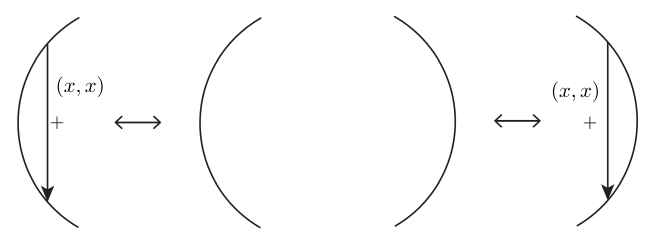} \includegraphics{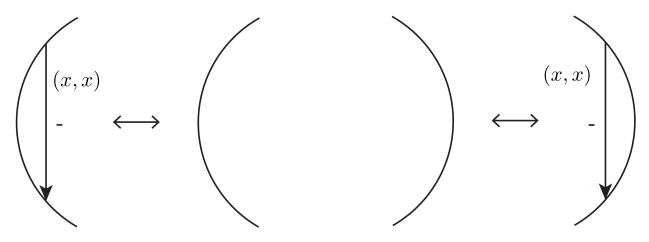}}\]

The RII moves involve introduction or removal of pairs of arrows with opposite
signs and the same pair of biquandle colors, either crossing or not.
These moves impose two axioms on the weights. 
\[\includegraphics{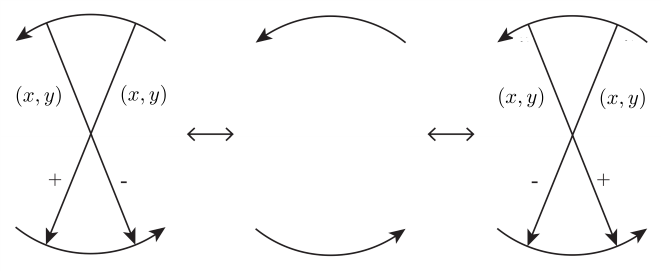}\]
For the RII moves where the arrows cross, we need the crossing weight to be 
zero; since these two arrows both have biquandle colors $(x,y)$, we set 
\[\phi((x,y),(x,y))=0\]
and obtain axiom (ii). For the RII moves where the arrows don't cross, the
weight contributions from the two arrows of the move on both sides are zero.
\[\includegraphics{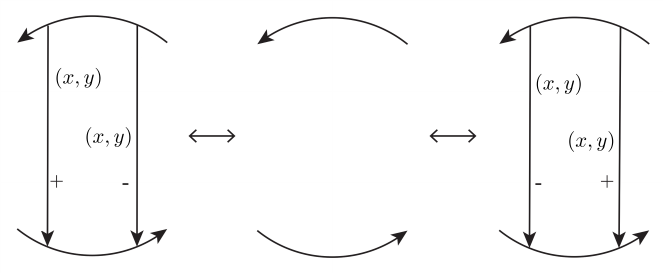}\]

More generally, the two arrows of either type of RII move will both cross the 
same set of arrows but with opposite signs
\[\includegraphics{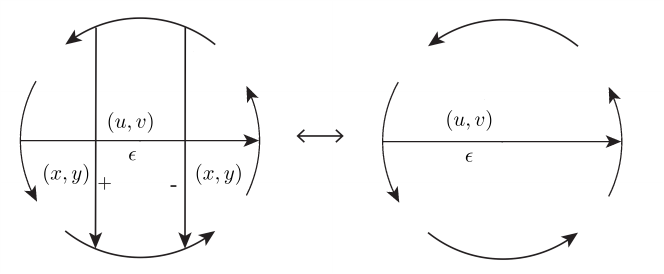}\]
so we arrange for the weights from the positive and negative signed arrows 
of the RII move to cancel by our weighting rule.

Our chosen generating set of Reidemeister moves contains only one RIII move,
but the two cyclic orderings of the three strands of the move impose different
conditions on biquandle arrow weights. Labeling the strands as shown
\[\includegraphics{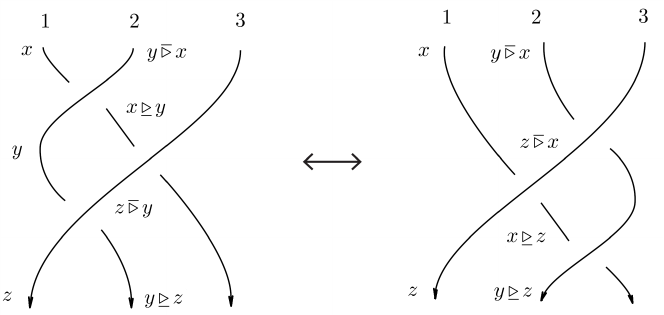}\]
we have one case 
\[\includegraphics{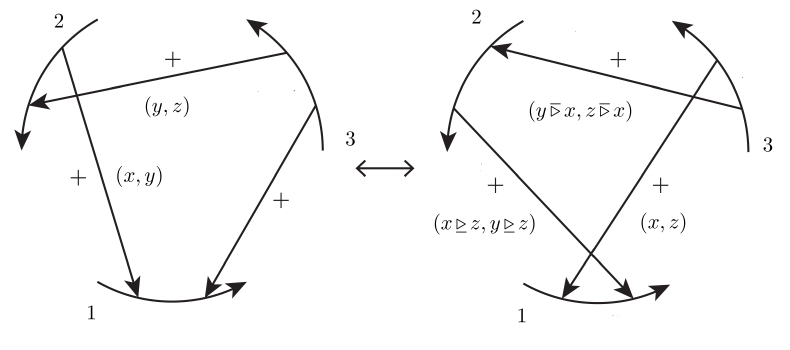}\]
which gives us axiom (iii) and the other
\[\includegraphics{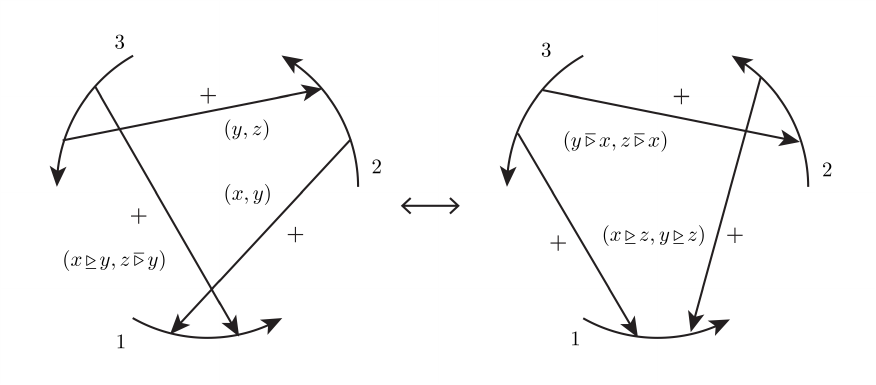}\]
which yields axiom (iv).

Finally, we must consider the crossings of other arrows with the arrows 
involved 
in the Reidemeister III moves. We note that such arrows cross a pair of arrows 
on each side of the move and that it suffices to consider positive arrows since 
negative arrows yield equivalent equations. We further note that the two cyclic 
orderings of the move produce the same requirements, namely the conditions in
axiom (v). We illustrate the case of the first equation in axiom (v); the 
others are similar.
\[\includegraphics{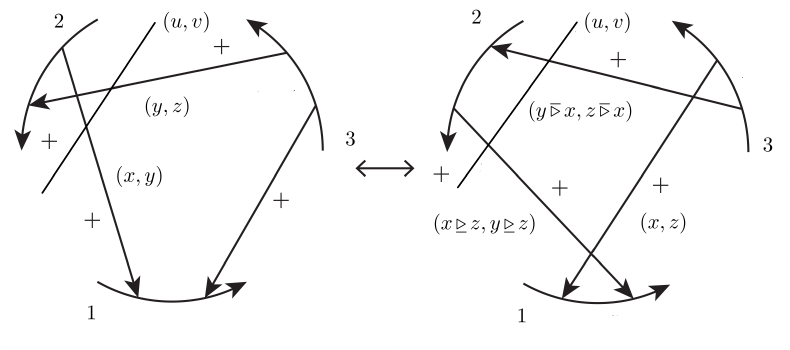}\]

Since this set of moves is a generating set, it follows that that the list of 
axioms obtained from it suffices to guarantee invariance.
\end{proof}

We can represent a biquandle arrow weight as a 4-tensor or matrix of matrices 
with values in the coefficient ring such that $\phi((i,j),(k,l))$ is the row 
$k$ column $l$ entry in the matrix in row $i$ column $j$.

\begin{example}\label{ex:1}
Our \texttt{python} computations indicate that the 4-tensor
\[\left[
\begin{array}{rr}
\left[\begin{array}{rr} 0 & 3 \\ 3 & 0\end{array}\right] &
\left[\begin{array}{rr} 3 & 0 \\ 6 & 9\end{array}\right] \\ & \\
\left[\begin{array}{rr} 3 & 6 \\ 0 & 9 \end{array}\right] &
\left[\begin{array}{rr} 0 & 9 \\ 9 & 0\end{array}\right] 
\end{array}\right]\]
defines a biquandle arrow weight over the biquandle structure on the set 
$X=\{1,2\}$ specified by the operation tables
\[
\begin{array}{r|rr}
\utr & 1 & 2 \\ \hline
1 & 2 & 2 \\
2 & 1 & 1
\end{array}
\quad 
\begin{array}{r|rr}
\otr & 1 & 2 \\ \hline
1 & 2 & 2 \\
2 & 1 & 1
\end{array}
\]
with coefficient group $A=\mathbb{Z}_{12}$. Then for instance we have
$\phi((1,2),(1,1))=3$ and $\phi((1,1),(2,2))=0.$
\end{example}

\begin{definition}
Let $X$ be a finite biquandle and $W$ a biquandle arrow weight with values
in an abelian group $A$. We define the \textit{biquandle arrow weight 
multiset} for Gauss diagram $G$ to be the multiset of $\Sigma_D$ values over 
the set of biquandle colorings $D$ of $G$,
\[\Phi_X^{W,M}(G)=\{\Sigma_D\ |\ D\in\mathrm{Hom}(\mathcal{B}(G),X)\}.\] 
We define the \textit{biquandle arrow weight polynomial} of $G$ to be the 
expression
\[\Phi_X^W(G)=\sum_{D\in\mathrm{Hom}(\mathcal{B}(G),X)} u^{S_D}\]
for a formal variable $u$.
\end{definition}

\begin{corollary}
$\Phi_X^{W,M}(G)$ and $\Phi_X^W(G)$ are invariants of oriented virtual (and 
hence classical) knots and links.
\end{corollary}

\begin{proof}
This follows immediately from Proposition \ref{prop:main}.
\end{proof}

\section{\Large\textbf{Examples}}\label{E}

In this section we collect a few examples of the new invariants.

\begin{example}\label{ex2}
Let us start with a basic illustration of how to compute the invariant. 
Let $X$ be the biquandle structure on the set $X=\{1,2\}$ defined by the
operation tables 
\[
\begin{array}{r|rr}
\utr & 1 & 2 \\ \hline
1 & 2 & 2 \\
2 & 1 & 1
\end{array}
\quad 
\begin{array}{r|rr}
\otr & 1 & 2 \\ \hline
1 & 2 & 2 \\
2 & 1 & 1
\end{array}
\]
and let $W$ be the biquandle arrow weight with coefficients in $A=\mathbb{Z}_8$
given by the 4-tensor 
\[\left[
\begin{array}{rr}
\left[\begin{array}{rr} 0 & 2 \\ 6 & 4\end{array}\right] &
\left[\begin{array}{rr} 2 & 0 \\ 0 & 2\end{array}\right] \\ & \\
\left[\begin{array}{rr} 6 & 0 \\ 0 & 6 \end{array}\right] &
\left[\begin{array}{rr} 4 & 2 \\ 6 & 0\end{array}\right] 
\end{array}\right]\]
with coefficients in $\mathbb{Z}_8$. 
Let us consider the virtual knot 4.72, given by the Gauss diagram
\[\includegraphics{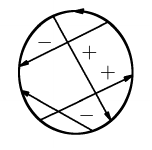}.\]
There are two $X$-colorings of this diagram,
\[\includegraphics{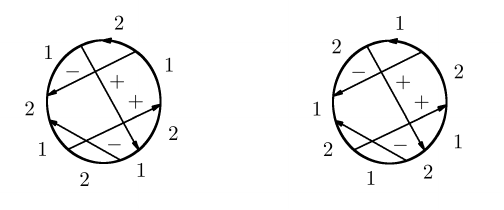}.\]

For the coloring on the left we have biquandle arrow weight
\[\phi((1,1,+),(2,1,-))+\phi((2,2,+),(1,2,-))+\phi((2,2,+),(1,1,+))=-6-2+4=4\]
and on the right we have
\[\phi((2,2,+),(1,2,-))+\phi((1,1,+),(2,1,-))+\phi((1,1,+),(2,2,+))=-2-6+4=4.\]
Then the multiset version of the invariant is 
\[\Phi_X^{W,M}(4.72)=\{4,4\}\]
and the polynomial version is
\[\Phi_X^{W}(4.72)=2u^4.\]
We note that since the unknot $0_1$ has values respectively 
$\Phi_X^{W,M}(0_1)=\{0,0\}$ and $\Phi_X^{W}(0_1)=2$, this example shows that
the invariant is nontrivial since it distinguishes 4.72 from the unknot.
Moreover, this example shows that the invariant is a proper enhancement,
i.e., it determines but is not determined by the biquandle counting invariant.
\end{example}

\begin{example}
Let $X$ be the biquandle in Example \ref{ex2}. We randomly selected two 
biquandle arrow weights with coefficients in $\mathbb{Z}_4$ 
\[W_1=\left[\begin{array}{rr}
\left[\begin{array}{rr} 0 & 1 \\ 1 & 0 \end{array}\right] &
\left[\begin{array}{rr} 1 & 0 \\ 2 & 3 \end{array}\right] \\
 & \\
\left[\begin{array}{rr} 1 & 2 \\ 0 & 3\end{array}\right] &
\left[\begin{array}{rr} 0 & 3 \\ 3 & 0\end{array}\right] 
\end{array}\right]
\quad
W_2=\left[\begin{array}{rr}
\left[\begin{array}{rr} 0 & 1 \\ 3 & 2 \end{array}\right] &
\left[\begin{array}{rr} 1 & 0 \\ 0 & 1 \end{array}\right] \\
 & \\
\left[\begin{array}{rr} 3 & 0 \\ 0 & 3 \end{array}\right] &
\left[\begin{array}{rr} 2 & 1 \\ 3 & 0\end{array}\right] 
\end{array}\right]
\]
and computed the
invariant values for the virtual knots with up to four classical crossings
in Jeremy Green's table at the knot atlas \cite{KA} as shown in the table.
\[\begin{array}{r|l}
\Phi_X^{W_1}(K) & K \\ \hline
2 & 3.1, 3.5, 3.6, 3.7, 4.1, 4.2, 4.3, 4.6, 4.7, 4.8, 4.10, 4.12, 4.13, 4.16, 4.17, 4.19, 4.21, 4.23, 4.24, 4.25, 4.26, \\ & 
4.29, 4.30, 4.33, 4.34, 4.37, 4.38, 4.39, 4.40, 4.43, 4.46, 4.47, 4.50, 4.51, 4.53, 4.55, 4.56, 4.60, 4.61, \\ & 
4.62, 4.63, 4.64, 4.69, 4.75, 4.76, 4.77, 4.79, 4.80, 4.85, 4.86, 4.89, 4.90, 4.91, 4.93, 4.96, 4.97, 4.98, \\ & 
4.99,  4.100,  4.102, 4.103, 4.105, 4.106, 4.107, 4.108\\
2u^2 & 2.1, 3.2, 3.2, 3.4, 4.4, 4.5, 4.9, 4.11, 4.14, 4.15, 4.18, 4.20, 4.22, 4.27, 4.28, 4.31, 4.32, 4.35, 4.36, 4.41, \\ &
4.42, 4.44, 4.45, 4.48, 4.49, 4.52, 4.54, 4.57, 4.58, 4.59, 4.65, 4.66, 4.67, 4.68, 4.70, 4.71, 4.72, 4.73, \\ &
4.74, 4.78, 4.81, 4.82, 4.83, 4.84, 4.87, 4.88, 4.92, 4.94, 4.95, 4.101, 4.104
\end{array}\]
\[\begin{array}{r|l}
\Phi_X^{W_2}(K) & K \\ \hline
2 & 2.1, 3.1, 3.2, 3.3, 3.4, 3.5, 3.6, 3.7, 4.1, 4.2, 4.3, 4.4, 4.5, 4.6, 4.7, 4.8, 4.9, 4.10, 4.11, 4.12, 4.13, 4.14, \\ &
4.15, 4.16, 4.17, 4.18, 4.19, 4.20, 4.21, 4.22, 4.23, 4.24, 4.25, 4.26, 4.27, 4.28, 4.43, 4.44, 4.45, 4.46, \\ & 
4.47, 4.48, 4.49, 4.50, 4.51, 4.52, 4.53, 4.54, 4.55, 4.56, 4.73, 4.74, 4.75, 4.76, 4.77, 4.78, 4.79, 4.80, \\ & 
4.81, 4.82, 4.83, 4.84, 4.85, 4.86, 4.87, 4.88, 4.89, 4.90, 4.91, 4.92, 4.93, 4.94, 4.95, 4.96, 4.97, \\ & 
4.98, 4.99, 4.100, 4.101, 4.102, 4.103, 4.104, 4.105, 4.106, 4.107, 4.108\\
2u^2 & 4.29, 4.30, 4.31, 4.32, 4.33, 4.34, 4.35, 4.36, 4.37, 4.38, 4.39, 4.40, 4.41, 4.42, 4.57, 4.58, 4.59, 4.60, \\ &
4.61,  4.62, 4.63, 4.64, 4.65, 4.66, 4.67, 4.68, 4.69, 4.70, 4.71, 4.72
\end{array}\]
\end{example}

\begin{example}
Let $X$ be the biquandle structure on $\{1,2,3\}$ given by the operation tables
\[
\begin{array}{r|rrr}
\utr & 1 & 2 & 3 \\ \hline
1 & 2 & 2 & 2 \\
2 & 3 & 3 & 3 \\
3 & 1 & 1 & 1
\end{array}
\quad 
\begin{array}{r|rrr}
\otr & 1 & 2 & 3 \\ \hline
1 & 2 & 2 & 2 \\
2 & 3 & 3 & 3 \\
3 & 1 & 1 & 1
\end{array}.
\]
Our \texttt{python} computations indicate that the 4-tensors
\[w_3=
\left[\begin{array}{rrr}
\left[\begin{array}{rrr} 0 & 1 & 2 \\ 2 & 0 & 1 \\ 1 & 2 & 0 \end{array}\right] &
\left[\begin{array}{rrr} 1 & 0 & 0 \\ 0 & 1 & 0 \\ 0 & 0 & 1 \end{array}\right] &
\left[\begin{array}{rrr} 2 & 0 & 0 \\ 0 & 2 & 0 \\ 0 & 0 & 2 \end{array}\right] \\
\left[\begin{array}{rrr} 2 & 0 & 0 \\ 0 & 2 & 0 \\ 0 & 0 & 2 \end{array}\right] &
\left[\begin{array}{rrr} 0 & 1 & 2 \\ 2 & 0 & 1 \\ 1 & 2 & 0 \end{array}\right] &
\left[\begin{array}{rrr} 1 & 0 & 0 \\ 0 & 1 & 0 \\ 0 & 0 & 1 \end{array}\right] \\
\left[\begin{array}{rrr} 1 & 0 & 0 \\ 0 & 1 & 0 \\ 0 & 0 & 1 \end{array}\right] &
\left[\begin{array}{rrr} 2 & 0 & 0 \\ 0 & 2 & 0 \\ 0 & 0 & 2 \end{array}\right] &
\left[\begin{array}{rrr} 0 & 1 & 2 \\ 2 & 0 & 1 \\ 1 & 2 & 0  \end{array}\right] \\
\end{array}\right]
\]
and
\[w_4=
\left[\begin{array}{rrr}
\left[\begin{array}{rrr} 0 & 2 & 1 \\ 1 & 0 & 2 \\ 2 & 1 & 0 \end{array}\right] &
\left[\begin{array}{rrr} 2 & 0 & 3 \\ 3 & 2 & 0 \\ 0 & 3 & 2 \end{array}\right] &
\left[\begin{array}{rrr} 1 & 3 & 0 \\ 0 & 1 & 3 \\ 3 & 0 & 1 \end{array}\right] \\
\left[\begin{array}{rrr} 1 & 3 & 0 \\ 0 & 1 & 3 \\ 3 & 0 & 1 \end{array}\right] &
\left[\begin{array}{rrr} 0 & 2 & 1 \\ 1 & 0 & 2 \\ 2 & 1 & 0 \end{array}\right] &
\left[\begin{array}{rrr} 2 & 0 & 3 \\ 3 & 2 & 0 \\ 0 & 3 & 2 \end{array}\right] \\
\left[\begin{array}{rrr} 2 & 0 & 3 \\ 3 & 2 & 0 \\ 0 & 3 & 2 \end{array}\right] &
\left[\begin{array}{rrr} 1 & 3 & 0 \\ 0 & 1 & 3 \\ 3 & 0 & 1 \end{array}\right] &
\left[\begin{array}{rrr} 0 & 2 & 1 \\ 1 & 0 & 2 \\ 2 & 1 & 0 \end{array}\right] \\
\end{array}\right]
\]
define biquandle arrow weight on $X$ with coefficients in 
$\mathbb{Z}_3$ and $\mathbb{Z}_6$ respectively, giving us the following table
of nontrivial invariant values for the prime  virtual knots with up to 4 classical 
crossings in the table at the Knot Atlas \cite{KA}:
\[
\begin{array}{r|l}
\Phi_X^{w_1}(K) & K \\ \hline
3u &  4.10, 4.13, 4.15, 4.18, 4.19, 4.20, 4.24, 4.29, 4.33, 4.34, 4.35, 4.38, 4.39, 4.40, 4.41, 4.42, 4.49, 4.50, \\ & 4.51, 4.57, 4.58, 4.70, 4.72, 4.78 \\
3u^2 & 4.11, 4.17, 4.22, 4.23, 4.32, 4.61, 4.62, 4.63, 4.66, 4.67, 4.68, 4.79 \\
\end{array}
\]
and
\[
\begin{array}{r|l}
\Phi_X^{w_2}(K) & K \\ \hline
3u &  4.17, 4.22, 4.23, 4.32, 4.61, 4.62, 4.67 \\
3u^2 & 4.13, 4.15, 4.20, 4.24, 4.29, 4.34, 4.38, 4.41, 4.42, 4.49, 4.51, 4.58, 4.72 \\
3u^3 & 2.1, 3.1, 3.2, 4.4, 4.5, 4.9, 4.14, 4.26, 4.27, 4.30, 4.37, 4.44, 4.47, 4.48, 4.52, 4.54, 4.60, 4.64, 4.69, \\ & 4.74, 4.80, 4.82, 4.84, 4.91, 4.93, 4.94, 4.102 \\
3u^4 & 4.11, 4.63, 4.66, 4.68, 4.79 \\
3u^5 & 4.10, 4.18, 4.19, 4.33, 4.35, 4.39, 4.40, 4.50, 4.57, 4.70, 4.78.
\end{array}
\]
Virtual knots not listed have the trivial invariant value of $\Phi_X^W(K)=3$.
\end{example}

\section{\Large\textbf{Questions}}\label{Q}

We conclude with some questions for future research.

The big question seems to be: what is the relationship between these 
invariants and the (bi)quandle 2-cocycle and 3-cocycle invariants from 
\cite{CJKLS} and later generalizations? Despite taking four biquandle 
colors as inputs, these biquandle arrow weights do not seem to be
identifiable as 4-cocycles in a obvious way.

What conditions analogous to cohomology make two biquandle arrow weights
define the same invariant? What is the geometric meaning of these invariants?

\bibliographystyle{abbrv}
\bibliography{sam-migiwa}

\bigskip

\noindent
\textsc{Department of Mathematical Sciences \\
Claremont McKenna College \\
850 Columbia Ave. \\
Claremont, CA 91711} 

\

\noindent
\textsc{ \\
Department of Materials Science and Engineering,\\
College of Engineering, Shibaura Institute of Technology,\\
307 Fukasaku, Minuma-ku, Saitama-shi, Saitama, 337-8570, Japan
}

\end{document}